\documentclass[np]{nyjm}
\usepackage[T1]{fontenc} 
\usepackage[utf8x]{inputenc}
\usepackage[usenames,dvipsnames]{xcolor}
\usepackage{wrapfig}
\usepackage{tkz-tab}
\usepackage{latexsym}
\usepackage{caption}
\usepackage{subcaption}
\usepackage{subfig}
\usepackage{graphicx, tikz}
\usetikzlibrary{calc}
\tikzstyle{stuff_fill}=[circle,draw,fill=black,font={A}]
\usepackage{amsmath}
\usepackage{marginnote}
\usepackage{color}
\newcommand\numberthis{\addtocounter{equation}{1}\tag{\theequation}}
\usepackage{caption}
\colorlet{lightpink}{pink!35!white}
\colorlet{lightyellow}{yellow!30!white}
\colorlet{lightpurple}{purple!20!white}
\colorlet{lightblue}{blue!15!white}

\numberwithin{equation}{section}
\theoremstyle{plain}
\newtheorem{thm}{Theorem}[section]
\newtheorem*{thm*}{Theorem}
\newtheorem*{ack}{Acknowledgement}
\newtheorem{lem}[thm]{Lemma}

\newtheorem{cor}{Corollary}
\theoremstyle{definition}

\theoremstyle{remark}
\newtheorem*{rem}{Remark}

\newcommand{\h}{\mathbb H}
\newcommand{\R}{\mathbb R}
\newcommand{\Z}{\mathbb Z}
\newcommand{\Q}{\mathbb Q}
\allowdisplaybreaks[4]

\newcommand{\BigOm}[1]{\ensuremath{\operatorname{O}\left(#1\right)}}

\newcommand{\BigOIe}[1]{\ensuremath{\operatorname{O}_{I,\epsilon}\left(#1\right)}}
\newcommand{\BigO}[1]{\ensuremath{\operatorname{O}_I\left(#1\right)}}

\newcommand{\defeq}{\mathrel{\mathop:}\mathrel{\mkern-1.2mu}=} 
\makeatletter
\def\pmod#1{\allowbreak\mkern10mu({\operator@font mod}\,\,#1)} 
\makeatother

\begin{document}

 \title{Geometry of Farey-Ford Polygons}
 
\author[Athreya]{Jayadev Athreya}
\address{Department of Mathematics, University of Illinois, 1409 West Green Street, Urbana, IL 61801, USA.}
\email{jathreya@illinois.edu} 
    \thanks{J.S.A partially supported by NSF CAREER grant 1351853; NSF grant
   DMS 1069153; and NSF grants DMS 1107452, 1107263, 1107367  ``RNMS: GEometric structures And Representation varieties" (the GEAR Network)."}

\author[Chaubey]{Sneha Chaubey}
\address{Department of Mathematics, University of Illinois, 1409 West Green Street, Urbana, IL 61801, USA.}
\email{chaubey2@illinois.edu}

\author[Malik]{Amita Malik}
\address{Department of Mathematics, University of Illinois, 1409 West Green Street, Urbana, IL 61801, USA.}
\email{amalik10@illinois.edu}
 
 \author[Zaharescu]{Alexandru Zaharescu}
\address{
Simion Stoilow Institute of Mathematics of the Romanian Academy, P.O. Box 1-764, RO-014700 Bucharest, Romania 
\textnormal{and}
Department of Mathematics, University of Illinois, 1409 West Green Street, Urbana, IL 61801, USA.}
\email{zaharesc@illinois.edu}

\begin{abstract} 
The Farey sequence is a natural exhaustion of the set of rational numbers between 0 and 1 by finite lists. Ford Circles are a natural family of mutually tangent circles associated to Farey fractions: they are an important object of study in the geometry of numbers and hyperbolic geometry. We define two sequences of polygons associated to these objects, the Euclidean and hyperbolic \emph{Farey-Ford polygons}. We study the asymptotic behavior of these polygons by exploring various geometric properties such as (but not limited to) areas, length and slopes of sides, and angles between sides.

 \end{abstract}

\maketitle

\section{Introduction and Main Results} \emph{Ford circles}  were introduced by Lester R. Ford \cite{Ford} in 1938 and since then have appeared in several different areas of mathematics. Their connection with \emph{Farey fractions} is well known. Exploiting this connection between the two quantities we form the \emph{Farey-Ford polygons} associated to them. We study some statistics associated to both the Euclidean and hyperbolic geometry of Farey-Ford polygons:  in particular distances between vertices, angles, slopes of edges and areas.
The analysis of these quantities is an interesting study in itself as distances, angles
and areas are among the most common notions traditionally studied
in geometry. 
We use dynamical methods to compute limiting distributions, and analytic number theory to compute asymptotic behavior of moments for the \emph{hyperbolic} and \emph{Euclidean} distance between the vertices.

Below, we define Farey-Ford polygons (\S\ref{sec:ford}), describe the statistics that we study (\S\ref{sec:geomstat}) and state our main theorems (\S\ref{sec:moments:results} and \S\ref{sec:dist:results}). In \S\ref{sec:moments} we prove our results on moments, in \S\ref{sec:dist} we prove our distribution results, and in \S\ref{sec:geom} we compute the tail behavior of the statistics we study.

\subsection{Farey-Ford polygons}\label{sec:ford} Recall the definition of \emph{Ford circles}: The Ford circle $$\mathcal C_{p/q}: =\left\{z \in \mathbb C: \left| z - \left(\frac p q  + i \frac{1}{2q^2}\right)\right| = \frac{1}{2q^2}\right\}$$ is the circle in the upper-half plane tangent to the point $\left(p/q, 0\right)$ with diameter $1/q^2$ (here and in what follows we assume that $p$ and $q$ are relatively prime). 

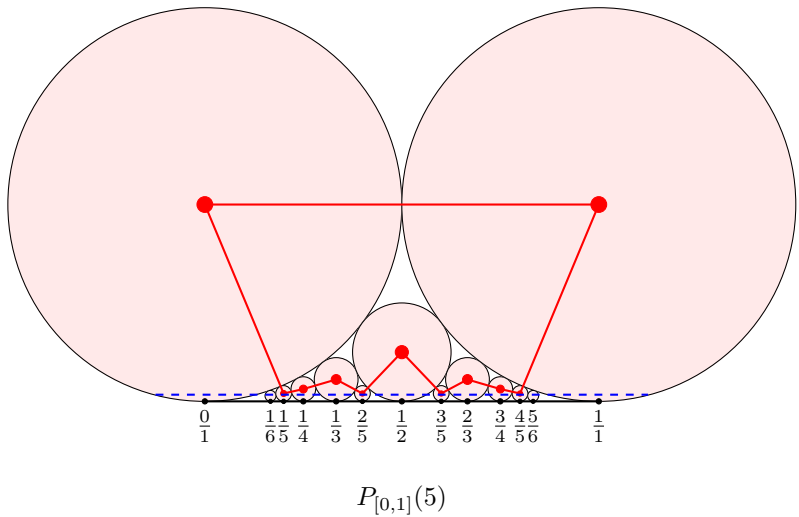
\begin{figure}[h!]
\begin{subfigure}[b]{.67\textwidth}
\centering
\resizebox{\linewidth}{!}{
\begin{tikzpicture}[scale=5.30]

 \def \firstcircle {(0, 1/2) circle(1/2)}
  \def \newcircle {(1/6, 1/72) circle(1/72)}
  \def \firstttcircle{(1/5, 1/50) circle(1/50)}
  \def \firsttcircle {(1/4, 1/32) circle(1/32)}
    \def \secondcircle {(1/3, 1/18) circle(1/18)}
    \def \thirddcircle {(2/5, 1/50) circle(1/50)}
      \def \thirdcircle  {(1/2, 1/8) circle(1/8)}
      \def \forthhhcircle   {(3/5, 1/50) circle(1/50)}
       \def \forthcircle   {(2/3, 1/18) circle(1/18)}
       \def \forthhcircle   {(3/4, 1/32) circle(1/32)}
       \def \fifthhcircle   {(4/5, 1/50) circle(1/50)}
       \def \fifthcircle   {(1, 1/2) circle(1/2)}
       \def \lastcircle {(5/6, 1/72) circle(1/72)}
       
       \def \vertex {(0,1/2) circle(1/50)}
       \def \newvertex {(1/5,1/50) circle(1/130)}
       \def \vert {(1/4,1/32) circle(1/100)}
        \def \vertexx {(1/3,1/18) circle(1/80)}
        \def \vertexxx {(1/2,1/8) circle(1/60)}
        \def \verttt {(2/3,1/18) circle(1/80)}
         \def \vertt {(3/4,1/32) circle(1/100)}
          \def \newvertexxxx {(4/5,1/50) circle(1/150)}
        \def \vertexxxx {(1,1/2) circle(1/50)}
        \def \lastvert {(2/5,1/50) circle(1/150)}
        \def \secondlastvert {(3/5,1/50) circle(1/150)}
        
        \def \firstcor {(0, 0) circle(1/150)}
        \def \newcor {(1/6, 0) circle(1/190)}
        \def \secondcor {(1/5,0) circle(1/180)}
        \def \thirdcor {(1/4,0) circle(1/150)}
        \def \forthcor {(1/3,0) circle(1/150)}
        \def \fifthcor {(2/5,0) circle(1/180)}
        \def \sixthcor {(1/2,0) circle(1/150)}
        \def \seventhcor {(3/5,0) circle(1/180)}
        \def \eighthcor {(2/3,0) circle(1/150)}
        \def \ninthcor {(3/4,0) circle(1/150)}
        \def \tenthcor {(4/5,0) circle(1/180)}
        \def \eleventhcor {(1,0) circle(1/150)}
        \def \lastcor {(5/6,0) circle(1/190)}
        

\filldraw[lightpink] \firstcircle;
    \draw \firstcircle ;
    \filldraw[lightpink] \newcircle;
    \draw \newcircle ;
    \filldraw[lightpink] \firstttcircle;
    \draw \firstttcircle ;
     \filldraw[lightpink] \firsttcircle;
    \draw \firsttcircle ;
     \filldraw[lightpink] \secondcircle;
     \draw \secondcircle ;
      \filldraw[lightpink] \thirddcircle;
      \draw \thirddcircle ;
      \filldraw[lightpink] \thirdcircle;
      \draw \thirdcircle ;
      \filldraw[lightpink] \forthhhcircle;
      \draw \forthhhcircle ;
      \filldraw[lightpink] \forthcircle;
      \draw \forthcircle ;
       \filldraw[lightpink] \forthhcircle;
      \draw \forthhcircle ;
       \filldraw[lightpink] \fifthhcircle;
      \draw \fifthhcircle ;
  \filldraw[lightpink] \fifthcircle;
      \draw \fifthcircle ;
       \filldraw[lightpink] \lastcircle;
      \draw \lastcircle ;
      \filldraw[red] \vertex;
      \filldraw[red] \newvertex;
      \filldraw[red] \vert;
     \filldraw[red] \vertexx;
     \filldraw[red] \vertexxx;
     \filldraw[red] \verttt;
     \filldraw[red] \vertt;
     \filldraw[red] \vertexxxx;
     \filldraw[red] \newvertexxxx;
       \filldraw[red] \secondlastvert;
       \filldraw[red] \lastvert;

\filldraw[black]\firstcor;
\filldraw[black]\newcor;
\filldraw[black]\secondcor;
\filldraw[black]\thirdcor;
\filldraw[black]\forthcor;
\filldraw[black]\fifthcor;
\filldraw[black]\sixthcor;
\filldraw[black]\seventhcor;
\filldraw[black]\eighthcor;
\filldraw[black]\ninthcor;
\filldraw[black]\tenthcor;
\filldraw[black]\eleventhcor;
\filldraw[black]\lastcor;

 \draw[ thick] (0, 0) -- (1, 0);

\node[align=left, below] at (0, 0){$\frac{0}{1}$};
\node[align=left, below] at (1/6, 0){$\frac{1}{6}$};
\node[align=left, below] at (1/5, 0){$\frac{1}{5}$};
\node[align=left, below] at (1/4, 0){$\frac{1}{4}$};
\node[align=left, below] at (1/3, 0){$\frac{1}{3}$};
\node[align=left, below] at (2/5, 0){$\frac{2}{5}$};
\node[align=left, below] at (1/2, 0){$\frac{1}{2}$};
\node[align=left, below] at (3/5, 0){$\frac{3}{5}$};
\node[align=left, below] at (2/3, 0){$\frac{2}{3}$};
\node[align=left, below] at (3/4, 0){$\frac{3}{4}$};
\node[align=left, below] at (4/5, 0){$\frac{4}{5}$};
\node[align=left, below] at (5/6, 0){$\frac{5}{6}$};
\node[align=left, below] at (1, 0){$\frac{1}{1}$};

\draw[red, thick] (0,1/2) -- (1/5,1/50)--(1/4,1/32)--(1/3, 1/18)--(2/5,1/50)--(1/2, 1/8)--(3/5,1/50)--(2/3,1/18)--(3/4,1/32)--(4/5,1/50)--(1, 1/2)--(0,1/2);
\draw[blue, thick, dashed] (-1/8,61/3600) -- (9/8,61/3600);
\node at (0.5,-0.25) {$P_{[0,1]}(5)$ };
\end{tikzpicture}
}
\end{subfigure}\\

\caption{Euclidean Farey-Ford Polygons}
\label{fgg}
\end{figure}

\noindent Given $\delta >0$ and $I=[\alpha, \beta] \subseteq [0,1]$, we consider the subset $$\mathcal{F}_{I, \delta}: = \left\{ \mathcal C_{p/q}: p/q \in I, \frac{1}{2q^2} \geq \delta\right\}$$ of Ford circles with centers above the horizontal line $y=\delta$ and with tangency points in $I$. Writing $$Q = \left\lfloor {(2\delta)}^{-1/2} \right\rfloor,$$ we have that the set of tangency points to the $x$-axis is $$\mathcal F(Q) \cap I,$$ where 
\[\mathcal F(Q):= \{ p/q: 0 \leq p\le q\le Q, (p,q)=1 \}
\]
 denotes the Farey fractions of order $Q$. We write $$\mathcal F(Q) \cap I : = \{ p_1/q_1 < p_2/q_2 \ldots <p_{N_I(Q)}/q_{N_I(Q)}\},$$ where $N_I(Q)$ is the cardinality of the above set (which grows quadratically in $Q$, so linearly in $\delta^{-1}$). We denote the circle based at $p_j/q_j$ by $C_j$ for $1\le j \le N_I(Q)$, and define the \emph{Farey-Ford polygon} $P_I(Q)$ of order $Q$ by 
 
 \begin{itemize}
 
 \item connecting $(0,1/2)$ to $(1,1/2)$ by a geodesic
 
 \medskip
 
 \item forming the ``bottom'' of the polygon by connecting the points $(p_j/q_j, 1/2{q_j}^2)$ and $(p_{j+1}/q_{j+1}, 1/2{q_{j+1}}^2)$ for $1\le j \le N_I(Q)-1$ by a geodesic. 
 
 \end{itemize}
 
 We call these polygons Euclidean (respectively hyperbolic) Farey-Ford polygons if the geodesics connecting the vertices are Euclidean (resp. hyperbolic). Figure \ref{fgg} shows the Euclidean Farey-Ford polygon $P_{[0,1]}(5)$.
\begin{rem}
For $ n \in \mathbb{N} $, consider the sequence $\langle f_n\rangle $ of functions $ f_n(x):[0,1]\rightarrow [0,1] $ defined as
\[f_n(x):=
\left\{
	\begin{array}{ll}
		\dfrac{1}{2q^2} & \mbox{if }  x=\frac{p}{q}, \ (p,q)=1, 1\le q\le n,\\
		0  & \mbox{otherwise} .
	\end{array}
\right.
\]

 In other words, $f_n(x) $ is non-zero only when $x$ is a Farey fraction of order $ n $, in which case $ f_n(x) $ denotes the Euclidean distance from the real axis to the vertices of Farey-Ford polygons of order $ n $. As we let $ n $ approach infinity, the sequence $ \langle f_n(x) \rangle $ converges and its limit is given by $\tfrac12(f(x))^2$, where $f(x)$ is Thomae's function defined on $[0,1]$ as follows: 
 \[f(x):=
 \left\{
 \begin{array}{ll}
 \frac{1}{q}  & \text{ if } x =\frac{p}{q}, \ (p,q)=1, q> 0\\ 
 0 & \text{ otherwise},
 \end{array}
 \right.
 \] 
 which is continuous at every irrational point in $ [0,1] $ and discontinuous at all rational points in $ [0,1].$
\end{rem}

\subsection{Geometric Statistics}\label{sec:geomstat} Given a Farey-Ford polygon $P_I(Q)$, for $1\le j\le N_I(Q)-1$, let $$ x_j  \defeq \frac{q_j}{Q}, \ y_j \defeq \frac{q_{j+1}}{Q}.$$  For $1\le j\le N_I(Q)-1$, we have $0 < x_j, y_j \le 1$. And, \[x_j+x_{j+1}=\frac{q_j+q_{j+1}}{Q}>1,\] since 
between any two Farey fractions $\frac{p_j}{q_j}$ and $\frac{p_{j+1}}{q_{j+1}},$ of order Q,
there is a Farey fraction given by $\frac{p_j+p_{j+1}}{q_j+q_{j+1}}$ which is not in $\mathcal{F}(Q) \cap I$ as $\frac{p_j}{q_j}$ and $\frac{p_{j+1}}{q_{j+1}}$ are neighbours in $\mathcal{F}(Q) \cap I$.

Let $$\Omega \defeq \{ (x,y) \in (0, 1]^2: x+y >1\},$$ so we have $(x_j, y_j) \in \Omega$ for $1\le j \le N_I(Q)-1$. A \emph{geometric statistic} $F$ is a measurable function $F: \Omega \rightarrow \R$. Given a geometric statistic F, we define the \emph{limiting distribution for Ford circles} $G_{F}(t)$ of F (if it exists) to be the limiting proportion of Ford circles for which $F$ exceeds $t$, that is, $$G_F(t) \defeq \lim_{\delta \rightarrow 0} \frac{ \#\{1 \le j \le N_I(Q): F(x_j, y_j) \geq t\}}{N_I(Q)}.$$ Similarly, we define the \emph{moment} of $F$ by 
\[ M_{F, I, \delta}(\mathcal F) : = \frac{1}{|I|}\sum_{j=1}^{N_I(Q)-1}| F(x_j, y_j)|.  \numberthis\label{Geomstat} \]
We describe how various natural geometric properties of Farey-Ford polygons can be studied using this framework.

\subsubsection{Euclidean distance} \label{euc} 
Given consecutive circles $C_j$ and $C_{j+1}$ in $\mathcal{F}_{I, \delta}$, we consider the \emph{Euclidean} distance $d_j$ between their centers $O_j$ and $O_{j+1}$.  Note that since consecutive circles are mutually tangent, the Euclidean geodesic connecting the centers of the circles passes through the point of tangency, and has length equal to the sum of the radii, that is, $$d_j  = \frac{1}{2q_j^2} + \frac{1}{2q_{j+1}^2}.$$ Therefore, we have $$Q^2 d_j = \frac{1}{2x_j^2} + \frac{1}{2y_j^2}.$$ Thus, by letting $F(x_j,y_j) = \frac{1}{2{x_j}^2} + \frac{1}{2{y_j}^2}$, we have expressed the Euclidean distance (appropriately normalized) as a geometric statistic.

In ~\cite{CMZ}, the latter three authors computed averages of first and higher moments of the Euclidean distance $d_j$, in essence, they computed explicit formulas for 
\[ \frac{1}{N_I(Q)}{\int_X}^{2X}\sum_{j=1}^{N_I(Q)-1}d_j^k  ~d(2 \delta)^{-1/2} \text{   for } k\ge 1,\]
where $X$ was a large real number.

\subsubsection{Slope} Let $t_j$ denote the Euclidean slope of the Euclidean geodesic joining the centers $O_j$ and $O_{j+1}$ of two consecutive circles $C_j$ and $C_{j+1}.$ It is equal to
\[t_j= \frac{1}{2}\left( \frac{q_i}{q_{i+1}}-\frac{q_{i+1}}{q_{i}}\right).\] We associate the geometric statistic 
$\displaystyle{F(x_j,y_j)=\frac{1}{2} \left( \frac{x_j}{ y_j} - \frac{y_j}{x_j} \right)}$ to it.

\subsubsection{Euclidean angles} Let $\theta_j$ be the angle between the Euclidean geodesic joining $O_j$ with $(p_j/q_j,0)$ and the Euclidean geodesic joining $O_j$ with the point of tangency of the circles $C_j$ and $C_{j+1}$. In order to compute $\theta_j$, we complete the right angled triangle with vertices $O_j$, $O_{j+1}$ and the point lying on the Euclidean geodesic joining $O_j$ and $(p_j/q_j,0).$ Then,

\[\tan\theta_j
=
\left\{
	\begin{array}{ll}
		\frac{2q_j q_{j+1}}{q^2_{j+1}-q_j^2}  & \mbox{if }  q_{j+1}>q_j,\\
		\frac{2q_j q_{j+1}}{q_j^2-q^2_{j+1}} & \mbox{if } q_{j+1}< q_j.
	\end{array}
\right.
\]
Thus, the geometric statistic associated to $\tan\theta_j$ is given by 
\[F(x_j,y_j)
=
\left\{
	\begin{array}{ll}
		\frac{2x_j y_j}{y_j^2-x_j^2} & \mbox{if }  y_j>x_j,\\
	 \frac{2x_jy_j}{x_j^2-y_j^2} & \mbox{if } y_j<x_j.
	\end{array}
\right. \numberthis\label{EucAng}
\]

\subsubsection{Euclidean area} Let $A_j$ denote the area of the trapezium whose sides are formed by the Euclidean geodesics joining the points $\left(\dfrac{p_j}{q_j}, 0\right)$ and $O_j$;  $\left(\dfrac{p_{j+1}}{q_{j+1}}, 0\right)$ and $O_{j+1}$; $O_{j}$ and $O_{j+1}$; $\left(\dfrac{p_j}{q_j}, 0\right)$ and $\left(\dfrac{p_{j+1}}{q_{j+1}}, 0\right)$. This gives $\displaystyle{A_j=\frac{1}{4q_j^3q_{j+1}}+\frac{1}{4q_jq_{j+1}^3}}.$ The geometric statistic (normalized) in this case is
\[F(x_j,y_j)= \frac{1}{4x_j^3y_{j}}+\frac{1}{4x_jy_{j}^3}.\]

\subsubsection{Hyperbolic distance} Given consecutive circles in $\mathcal{F}_{I, \delta}$ we let $\rho_j$ denote the \emph{hyperbolic} distance between their centers $O_j$ and $O_{j+1}$, where we use the standard hyperbolic metric on the upper half plane $\h$ defined as $$ds : = \frac{\sqrt{dx^2 + dy^2}}{y}.$$ A direct computation shows that
\[\sinh \frac{\rho_j}{2}= \frac{q_j}{2q_{j+1}}+\frac{q_{j+1}}{2q_j} = \frac{x_j}{2y_j} + \frac{y_j}{2x_j} = \frac{1}{2} \left( \frac {x_j} {y_j} + \frac {y_j} {x_j} \right)\] Letting $\displaystyle{F(x_j,y_j)=\frac{1}{2} \left( \frac{x_j}{ y_j} +\frac{y_j}{x_j} \right)}$, we have expressed ($\sinh$ of half of) the hyperbolic distance as a geometric statistic.

\subsubsection{Hyperbolic angles} Let $\alpha_j$ denote the angle between the hyperbolic geodesic parallel to the $y$-axis, passing through the center $O_j$, and the hyperbolic geodesic joining the centers $O_j$ and $O_{j+1}$. This angle can be calculated by first computing the center of the hyperbolic geodesic joining $O_j$ and $O_{j+1}$. Note that the center of this geodesic lies on the x-axis. Let us denote it by $(c_j,0)$. Now, $\alpha_j$ is also the acute angle formed between the line joining $(c_j,0)$ and $O_j$, and the x-axis. Thus $\tan\alpha_j$ is given by
\[\tan\alpha_j=\frac{4q_j^5}{q_{j+1}^3} - 4q_jq_{j+1}+\frac{16q_j^3}{q_{j+1}}.\]
Consequently the geometric statistic (normalized) associated to it is \[F(x_j,y_j)=\frac{4x_j^5}{y_{j}^3} - 4x_jy_{j}+\frac{16x_j^3}{y_{j}}.\]

\subsection{Moments}\label{sec:moments:results}The following results concern the growth of the moments $M_{F, I, \delta}$ for the geometric statistics expressed in terms of Euclidean distances, hyperbolic distances and angles. 
\begin{thm}\label{theorem:eucl:moment}
For any real number $\delta>0,$ and any interval $ I=[\alpha, \beta]\subset [0,1]$, with $\alpha,\beta\in\mathbb{Q}$ and $F(x,y)=\frac{1}{2x^2}+\frac{1}{2y^2}$,
\begin{align*}
M_{F,I,\delta}(\mathcal{F})=\frac{6}{\pi^2}Q^2\log Q+D_{I}Q^2+\BigO{Q\log Q},
\end{align*}
where $M_{F,I,\delta}(\mathcal{F})$ is defined as in \eqref{Geomstat}, $Q$ is the integer part of $\frac{1}{\sqrt{2\delta}}$ and $D_I$ is a constant depending only on the interval $I$. Here the implied constant in the Big O-term depends on the interval I.
\end{thm}

\begin{thm}\label{theorem:hyp:moment} 
For any real number $\delta>0,$ any interval $ I=[\alpha, \beta]\subset [0,1]$ and $F(x,y)=\frac{1}{2}\left(\frac{x}{y}+\frac{y}{x}\right)$,
\begin{align*}
M_{F,I,\delta}(\mathcal{F})= \frac{9}{2\pi^2}Q^2 +\BigOIe{Q^{7/4+\epsilon}},
\end{align*}
where $M_{F,I,\delta}(\mathcal{F})$ is defined as in \eqref{Geomstat}, $C_I$ is a constant depending only on the interval $I$, $\epsilon$ is any positive real number, and $Q$ is the integer part of $\frac{1}{\sqrt{2\delta}}$. The implied constant in the Big O-term depends on the interval I and $\epsilon$.
\end{thm}

\begin{thm}\label{theorem:angle:moment} 
For any real number $\delta>0,$ any interval $ I=[\alpha, \beta]\subset [0,1]$ and $F(x,y)$ as in \eqref{EucAng}, we have,
\begin{align*}
M_{F,I,\delta}(\mathcal{F})= \frac{12}{\pi}Q^2 +\BigO{Q\log Q},
\end{align*}
where $M_{F,I,\delta}(\mathcal{F})$ is defined as in \eqref{Geomstat}, and $Q$ is the integer part of $\frac{1} {\sqrt{2\delta}}$. The implied constant in the Big O-term depends on the interval I.
\end{thm}

\subsection{Distributions}\label{sec:dist:results} The equidistribution of a certain family of measures on $\Omega$ yields information on the distribution of individual geometric statistics and also on the \emph{joint} distribution of any finite family of \emph{shifted} geometric statistics. In particular, we can understand how all of the above statistics correlate with each other, and can even understand how the statistics observed at shifted indices correlate with each other. First, we record the result on equidistribution of measures.

Let $$\rho_{Q, I} := \frac 1 N  \sum_{i=1}^N \delta_{(x_j, y_j)}$$ be the probability measure supported on the set $\{(x_j, y_j)=\left(\frac{q_j}{Q},\frac{q_{j+1}}{Q}\right): 1 \le j \le N\}$. 
We have~\cite[Theorem 1.3]{Athreya} (see also~\cite{KZ:1997}, \cite{Marklof}):
\begin{thm}\label{theorem:equi}The measures $\rho_{Q,I}$ equidistribute as $Q \rightarrow \infty$. That is, $$\lim_{Q \rightarrow \infty} \rho_{Q, I} = m,$$ where $dm = 2 dx dy$ is the Lebesgue probability measure on $\Omega$ and the convergence is in the weak-$*$ topology.
\end{thm}

As consequences, we obtain results on individual and joint distributions of geometric statistics.
\begin{cor}\label{thm:stat:ind} Let $F: \Omega \rightarrow \R^+$ be a geometric statistic and let $m$ denote the Lebesgue probability measure on $\Omega$. Then for any interval $I \subset [0, 1]$, the limiting distribution $G_{F}(t)$ exists, and is given by $$G_F(t) = m\left( F^{-1}(t, \infty) \right).$$

\end{cor}

In order to obtain results on joint distributions, we will need the following notations. Let $T: \Omega \rightarrow \Omega$ be the \emph{BCZ map}, $$T(x,y) := \left (y, -x + \left \lfloor \frac{1+x}{y} \right \rfloor y \right).$$ A crucial observation, due to Boca, Cobeli and one of the authors ~\cite{BCZ} is that $$T(x_j, y_j) = (x_{j+1}, y_{j+1}).$$ 

\noindent Let $\mathbf F = \{F_1, \ldots, F_k\}$ denote a finite collection of geometric statistics, and let $\mathbf n = (n_1, \ldots, n_k) \in \Z^k$ and $\mathbf t = (t_1, \ldots, t_k) \in \R^k$. Define $G_{\mathbf {F,n}}(\mathbf t) $, the $(\mathbf {F, n})$-\emph{limiting distribution} for Ford circles by the limit (if it exists), $$G_{\mathbf {F, n}}(\mathbf t) \defeq \lim_{\delta \rightarrow 0}  \frac{ \#\{1 \le j \le N_I(Q):  F_i(x_{j+n_i}, y_{j+n_i}) \geq t_i, 1 \le i \le k\}}{N_I(Q)},$$ where the subscripts are viewed modulo $N_I(Q)$. This quantity reflects the proportion of circles with specified behavior of the statistics $F_i$ at the indices $j + n_i$. Corollary~\ref{thm:stat:ind} is in fact a special ($k=1, \mathbf n = 0$) case of the following.

\begin{cor}\label{thm:stat:joint} For any finite collection of geometric statistics $\mathbf F = \{F_1, \ldots, F_k\}$, $\mathbf n = (n_1, \ldots, n_k) \in \Z^k, \mathbf t = (t_1, \ldots, t_k) \in \R^k$ the limiting distribution $G_{\mathbf {F,n}}(\mathbf t)$ exists, and is given by $$G_{\mathbf {F,n}} (\mathbf t) = m \left( \left(\mathbf F \circ T\right)^{-\mathbf n} (t_1, \ldots, t_k) \right),$$
where \[\left(\mathbf F \circ T\right)^{-\mathbf n}(t_1, \ldots, t_k) := \bigcap_{i=1}^k \left(F_i \circ T^{n_i}\right)^{-1}(t_i, \infty). \]
\end{cor}

\subsubsection{Tails}\label{sec:tails} For the geometric statistics listed in Section\ref{sec:geomstat}, one can also explicitly compute \emph{tail} behavior of the limiting distributions, that is, the behavior of $G_{F}(t)$ as $t\rightarrow \infty$. We use the notation $f(x)\sim g(x)$ as $x\rightarrow\infty$ to mean that
\[\lim_{x\rightarrow\infty}\frac{f(x)}{g(x)}=1,\]
and the notation $f(x)\asymp g(x)$  means that there exist positive constants $c_1, c_2$ and a constant $x_0$ such that for all $x> x_0$,
\[c_1 |g(x)|\le | f(x)| \le c_2|g(x)|.\]
We have the following result for the tail behavior of limiting distributions for various geometric statistics.

\begin{thm}\label{theorem:tails} Consider the geometric statistics associated to the Euclidean distance, hyperbolic distance, Euclidean slope, Euclidean area and hyperbolic angle. The limiting distribution functions associated to these have the following tail behavior:
	 \begin{enumerate}
\item {\normalfont(Euclidean distance)} for $F(x,y)=\frac{1}{2x^2}+\frac{1}{2y^2}, G_{F}(t)\sim\frac{1}{t}.$
\item {\normalfont(Hyperbolic distance)}
for $F(x,y)=\frac{1}{2}\left(\frac{y}{x}+\frac{x}{y}\right),\; G_{F}(t)\sim\frac{1}{2t^2} . $
\item {\normalfont (Euclidean slope)}
for $F(x,y)=\frac{1}{2} \left( \frac{x}{ y} - \frac{y}{x} \right),\; G_{F}(t)\asymp\frac{1}{t^2}.$
\item {\normalfont(Euclidean area)}
for $F(x,y)=\frac{1}{4} \left( \frac{1}{x^3y}+\frac{1}{xy^3} \right),\; G_{F}(t)\asymp\frac{1}{t^{2/3}}.$
\item {\normalfont(Hyperbolic angle)} for $F(x,y)=\frac{4x^5}{y^3} - 4xy+\frac{16x^3}{y},\; G_{F}(t)\asymp\frac{1}{t^{2/3}}.$
\end{enumerate}
\end{thm}

\section{Moments for Euclidean distance}\label{sec1:moments}

In this section, we prove Theorem~\ref{theorem:eucl:moment}. Recall that the Euclidean distance $d_j$ between the centers of two consecutive Ford circles $C_j$ and $C_{j+1}$ is given by \[d_j=\frac{1}{{2q_j}^2}+\frac{1}{{2q_{j+1}}^2}.\] 
As mentioned in section \ref{euc}, the geometric statistic $F$ expressed in terms of the Euclidean distance $d_j$ is
\[F(x_j,y_j)=Q^2d_j.\]
Thus, the moment of $F$ is given by
\[M_{F,I, \delta}(\mathcal{F})=\frac{Q^2}{|I|}\sum_{j=1}^{N_I(Q)-1} \left(\frac{1}{2q^{2}_j}+\frac{1}{2q^{2}_{j+1}}\right).\] 
\begin{proof}[Proof of Theorem \ref{theorem:eucl:moment}]
\begin{align*}
\frac{|I|}{Q^2}M_{E,I, \delta}(\mathcal{F})&
=\sum_{j=2}^{N_I(Q)}\frac{1}{q^{2}_j} + \frac{1}{2 q_1^{2}}-\frac{1}{2 q_{N_I(Q)}^{2}}
=  \frac{1}{2q_1^{2}}-\frac{1}{2q_{N_I(Q)}^{2}}+\sum_{1\le q\le Q}\frac{1}{q^{2}}\sum_{\substack{\alpha q< a \le\beta q \\ (a,q)=1}}1
 \\
 &=\frac{1}{2q_1^{2}}-\frac{1}{2q_{N_I(Q)}^{2}}+\sum_{q\le Q}\frac{1}{q^{2}}\sum_{\alpha q< a \le\beta q}1\sum_{d|(a,q)}
\mu(d)\\
 &=\frac{1}{2q_1^{2}}-\frac{1}{2q_{N_I(Q)}^{2}}+\sum_{q\le Q}\frac{1}{q^{2}}\sum_{d|q}\mu(d)\sum_{\frac{\alpha q}{d}< l \le\frac{\beta q}{d}}1\\
 &=\frac{1}{2q_1^{2}}-\frac{1}{2q_{N_I(Q)}^{2}}+\sum_{q\le Q}\frac{1}{q^{2}}\left(\sum_{d|q}\mu(d)\frac{(\beta-\alpha)q}{d}-
\sum_{d|q}\mu(d)\left(\left\{\frac{\beta q}{d}\right\}-\left\{\frac{\alpha q}{d}\right\}\right)\right)\\
&=|I|\sum_{q\le Q}\frac{\phi(q)}{q^{2}}-\sum_{d\le Q}\frac{\mu(d)}{d^{2}}\sum_{m\le\frac{Q}{d}}\left(\frac{\left\{\beta m\right\}
-\left\{\alpha m\right\}}{m^{2}}\right)+\left(\frac{1}{2q_1^{2}}-\frac{1}{2q_{N_I(Q)}^{2}}\right). \numberthis\label{Euc}
\end{align*}
Note that since $\alpha, \beta\in\Q$, for large $Q$ one can assume that $\alpha$ and $\beta$ are Farey fractions of order $Q$. Therefore, the quantity 
$\dfrac{1}{2q_1^{2}}-\dfrac{1}{2q_{N_I(Q)}^{2}}$ is a constant.
Now,
\begin{align*}
\sum_{m\le\frac{Q}{d}}\left(\frac{\left\{\beta m\right\}
-\left\{\alpha m\right\}}{m^{2}}\right)&=\sum_{m=1}^{\infty}\left(\frac{\left\{\beta m\right\}
-\left\{\alpha m\right\}}{m^{2}}\right) -\sum_{m>\frac{Q}{d}}\left(\frac{\left\{\beta m\right\}
-\left\{\alpha m\right\}}{m^{2}}\right)\\&= C_{I}+\BigO{\frac{d}{Q}}.
\end{align*}
This gives
\begin{align*}
\sum_{d\le Q}\frac{\mu(d)}{d^{2}}\left(\frac{\left\{\beta m\right\}
-\left\{\alpha m\right\}}{m^{2}}\right)=\frac{C_{I}}{\zeta(2)}+\BigO{\frac{\log Q}{Q}}.
\end{align*}
Combining this in \eqref{Euc} and the fact that
\[ \sum_{q\le Q}\frac{\phi(q)}{q^{2}}=\frac{6}{\pi^2}\log Q+
A+\BigOm{\frac{\log Q}{Q}}
, \numberthis\label{phi} \]
we obtain,

\begin{align*}
M_{E,I, \delta}(\mathcal{F})=\frac{6\log Q}{\pi^2}+D_{I}+\BigO{\frac{\log  Q}{Q}}.
\end{align*}
It only remains to prove \eqref{phi}.
\begin{align*}
\sum_{n\le x}\frac{\phi(n)}{n^2}&=\sum_{n\le x}\frac{1}{n^2}\sum_{d|n}\mu(d)\frac{n}{d}=\sum_{\substack{q,d \\ qd\le x}}\frac{\mu(d)}{qd^2} \\ &=\sum_{d\le x}\frac{\mu(d)}{d^2}\sum_{q\le x/d}\frac{1}{q}=\sum_{d\le x}\frac{\mu(d)}{d^2}\left(\log x-\log d+A+\BigOm{\frac{1}{x}}\right)\\&=\log x\left(\sum_{d=1}^{\infty}\frac{\mu(d)}{d^2}-\sum_{d>x}\frac{\mu(d)}{d^2}\right)-\left(\sum_{d=1}^{\infty}\frac{\mu(d)\log d}{d^2}-\sum_{d>x}\frac{\mu(d)\log d}{d^2}\right) \\& +
A\left(\sum_{d=1}^{\infty}\frac{\mu(d)}{d^2}-\sum_{d>x}\frac{\mu(d)}{d^2}\right)+\BigOm{\frac{1}{x}}\\&=
\frac{\log x}{\zeta(2)}+B+\BigOm{\frac{\log x}{x}}.
\end{align*}

This completes the proof of Theorem~\ref{theorem:eucl:moment}.
\end{proof}
\section{Moments for hyperbolic distance and angles}\label{sec:moments}

\subsection{Hyperbolic distance}

In this section, we prove Theorem~\ref{theorem:hyp:moment}. Recall that $\rho_j$ denotes the hyperbolic distance between consecutive centers of Ford circles $O_j$ and $O_{j+1}$, and
\[\sinh \frac{\rho_j}{2}= \frac{q_j}{2q_{j+1}}+\frac{q_{j+1}}{2q_j}. \numberthis \label{xy1}\]
In order to compute the total sum of the distances $\sinh \frac{\rho_j}{2}$,
\begin{align*}
M_{H,I,\delta}(\mathcal{F})&=\frac 1 {|I|}\sum_j \sinh  \frac{\rho_j}{2} =\frac 1 {|I|}\sum_{\substack{q,q' \text{ neighbours }\\ \text{ in }  \mathcal{F}_I(Q)}} \left(\frac{q}{2q'}+\frac{q'}{2q} \right), \numberthis \label{cases}
\end{align*}
we use the following key analytic lemma.
\begin{lem}\cite[Lemma 2.3]{BCZ}\label{lem1}
Suppose that $0<a<b$ are two real numbers, and $q$ is a positive integer. Let $f$ be a piecewise $C^1$ function on $[a,b].$ Then
\begin{align*}
\sum_{a<k\le b }\frac{\phi(k)}{k}f(k)= \frac{6}{\pi^2}\int_a^b f(x)~dx +\BigO{\log b \left(||f||_\infty + \int_a^b |f'(x)| ~dx \right)}.
\end{align*}
\end{lem}
\noindent We also use the Abel summation formula,\\
 \[\sum_{x < n \leq y} a(n)f(n)=A(y)f(y)-A(x)f(x)-\int_x^y A(t)f'(t)~dt, \numberthis \label{Abel}\]
where $ a:\mathbb{N} \rightarrow \mathbb{C}$ is an arithmetic function, $0 < x < y$ are
real numbers,  $f : [x, y]\rightarrow\mathbb{C}$  is a function with continuous
derivative on $[x, y]$, and $A(t)\defeq \sum_{1 \leq n \leq t}a(n)$.
\begin{proof}[Proof of Theorem \ref{theorem:hyp:moment}]
Let $\bar{q'}$ denote the unique multiplicative inverse of $q'$ modulo $q$ in the interval $[1,q]$. Define \[a_q(q') =
\left\{
	\begin{array}{ll}
		1  & \mbox{if } (q,q')=1\text{ and } \ \bar{q'}\in[q\alpha, q\beta], \\
		0 & \mbox{otherwise },
	\end{array}
\right. 
.\]
Using Weil type estimates (\cite{Ester}, \cite{Hooley}, \cite{Weil}) for Kloosterman sums, the partial sum of $a_q(q')$ was estimated in \cite{BCZ}. More precisely, by
\cite[Lemma 1.7]{BCZ}, for any $\epsilon >0$,
 \[A_q(t)\defeq \sum_{1 \leq q' \leq t}a_q(q')=\frac{\phi(q)}{q}(t-1)|I|+\BigOIe{q^{1/2+\epsilon}}.\numberthis \label{A(t)}\] 
In order to estimate the sum in \eqref{cases}, we write it as
\[2|I|M_{H,I,\delta}(\mathcal{F})=\sum_ {\substack{q,q' \text{ neighbours }\\ \text{ in }  \mathcal{F}_I(Q) \\ q\ge q'}}\left(\frac{q}{q'}+\frac{q'}{q}\right)+
\sum_ {\substack{q,q' \text{ neighbours }\\ \text{ in }  \mathcal{F}_I(Q) \\ q< q'}}\left(\frac{q}{q'}+\frac{q'}{q}\right). \numberthis\label{sumpart}\]
Note that it suffices to estimate the first sum above because of the symmetry in $q$ and $q'$. And for the first sum, one has
\begin{align*}
\sum_ {\substack{q,q' \text{ neighbours }\\ \text{ in }  \mathcal{F}_I(Q) \\ q\ge q'}}\left(\frac{q}{q'}+\frac{q'}{q}\right)&=\sum_{\substack{{Q/2<q\le Q-L} \\ Q-q<q'\le q}} a_q(q')\left(\frac{q}{q'}+\frac{q'}{q}\right)+\sum_{\substack{{Q-L<q\le Q}\\ L<q'\le q}}a_q(q') \left(\frac{q}{q'}+\frac{q'}{q}\right) \\&+\sum_{\substack{Q-L<q\le Q \\Q-q<q'\le L}}a_q(q') \left(\frac{q}{q'}+\frac{q'}{q}\right)\\&=:S_1+S_2+S_3.
\numberthis\label{Mainhyperbolic}
\end{align*}
We begin with estimating the sum $S_1$.
\begin{align*}
S_1&= \sum_{ Q/2<q\le Q-L} \sum_{Q-q<q'\le q} a_q(q')\left(\frac{q}{q'}+\frac{q'}{q}\right)=:\sum_{Q/2<q\leq Q-L}\Sigma_1.
\end{align*}
Employing \eqref{A(t)}, \eqref{Abel} and the fact that $q>Q-q$, for the inner sum $\Sigma_1$ in $S_1$, we have
\begin{align*}
\Sigma_1&=A_q(q)\left(\frac{q}{q}+\frac{q}{q}\right)
-A_q(Q-q)\left(\frac{Q-q}{q}+\frac{q}{Q-q}\right)-\int_{Q-q}^q A_q(t)\left(\frac{1}{q}-\frac{q}{t^2}\right)~dt\\
	&=2 \left(\frac{\phi(q)}{q}(q-1)|I|+\BigOIe{q^{1/2+\epsilon}}\right)-\left(\frac{\phi(q)}{q}(Q-q-1)|I|+\BigOIe{q^{1/2+\epsilon}}\right) \\
	& \quad \left(\frac{Q-q}{q}+\frac{q}{Q-q}\right)-\int_{Q-q}^q\left(\frac{\phi(q)}{q}(t-1)|I|+\BigOIe{q^{1/2+\epsilon}}\right)\left(\frac{1}{q}-\frac{q}{t^2}\right)~dt\\
	&=|I|\frac{\phi(q)}{q}\left(2(q-1)-(Q-q-1)\left(
\frac{Q-q}{q}+\frac{q}{Q-q}\right)-\int_{Q-q}^q (t-1)\left(\frac{1}{q}-\frac{q}{t^2}\right)
~dt \right)\\
	& \quad+\BigOIe{q^{1/2+\epsilon}}\left(\frac{q}{Q-q}\right)\\
	&=|I|\frac{\phi(q)}{q}\left(2(q-1)-(Q-q-1)\left(\frac{Q-q}{q}+\frac{q}{Q-q}\right)+\frac{Q^2}{2 q}-\frac{Q}{q}\right. \\
	& \left. \quad-\frac{q}{Q-q}-q \log (Q-q)+q \log q-Q+3\right)+ \BigOIe{\frac{q^{3/2+\epsilon}}{Q-q}}\\
	& = -\frac{Q^2}{2 q}-q \log (Q-q)+q \log q+Q+ \BigOIe{\frac{q^{3/2+\epsilon}}{Q-q}}\numberthis\label{sigma1}
\end{align*}
Therefore, 
\begin{align*}
S_1=|I|\sum_{\frac{Q}{2}<q\le Q-L}\frac{\phi(q)}{q} g(q) +\BigOIe{\frac{Q^{5/2+\epsilon}}{L}},
\end{align*}
where $g(q):=-\frac{Q^2}{2 q}-q \log (Q-q)+q \log q+Q$.
Also, from \eqref{sigma1} and Lemma \ref{lem1}, we obtain 
\[S_1=\frac{6|I|}{\pi^2}\int_{Q/2}^{Q-L} g(x)~dx +\BigOIe{\log Q (||g||_\infty +\int_{Q/2}^{Q-L}|g'(x)|~dx)} +\BigOIe{\frac{Q^{5/2+\epsilon}}{L}}.\numberthis\label{S1} \]
And,
\begin{align*}
&\int_{Q/2}^{Q-L}g(q)=\frac{Q^2-(Q-L)^2}{2} \log \left(\frac{L}{Q-L}\right)+\frac{3 (Q-L) Q}{2}-\frac{3 Q^2}{4};
\numberthis\label{IntS1}\\
&||g||_\infty=\BigO{Q\log Q}= \int_{Q/2}^{Q-L} |g'(x)| ~dx. \numberthis\label{error2} 
\end{align*}
Therefore, from \eqref{S1}, \eqref{IntS1}, \eqref{error2} and Lemma \ref{lem1},
\begin{align*} 
S_1&=\frac{6|I|}{\pi^2} \left(\frac{Q^2-(Q-L)^2}{2} \log \left(\frac{L}{Q-L}\right)+\frac{3 (Q-L) Q}{2}-\frac{3 Q^2}{4}\right)+\BigOIe{\frac{Q^{5/2+\epsilon}}{L}}\\
&\quad +\BigO{Q\log Q}.\numberthis\label{FinalS1}
\end{align*}
Next, we estimate the sum $S_2$ in \eqref{Mainhyperbolic}. Recall that since $q'\leq q$, we have
 \begin{align*}
  S_2&=\sum_{Q-L<q\le Q}\sum_{L<q'\le q }a_q(q') \left(\frac{q}{q'}+\frac{q'}{q}\right)=\BigO{\sum_{Q-L<q\le Q}\sum_{L<q'\le q} \frac{q}{q'}}
\numberthis\label{S2}\\
	&=\BigO{\sum_{Q-L<q\leq Q}q\sum_{L<q'\le q}\frac{1}{q'}}=\BigO{QL\log Q}\numberthis\label{FinalS2}.
\end{align*}
Lastly, 
\begin{align*}
S_3&=\sum_{Q-L<q\le Q}\sum_{Q-q<q'\le L}a_q(q') \left(\frac{q}{q'}+\frac{q'}{q}\right)\\
	&=\BigO{\sum_{1\leq q'\le L}\sum_{Q-q'<q\le Q}\frac{q}{q'}}=\BigO{QL}, \numberthis\label{S3} 
\end{align*}
Setting $L=Q^{3/4}$, and from \eqref{FinalS1}, \eqref{FinalS2}, \eqref{S3}, and \eqref{Mainhyperbolic}, we obtain the first moment for hyperbolic distances as
\[ M_{H,I,\delta}(\mathcal{F})=\frac{9}{2\pi^2}Q^2+\BigOIe{Q^{7/4+\epsilon}}. \]
\end{proof}
\subsection{Angles}
 First we consider the case when $q_{j+1}$ is smaller than $q_j$. Let $\theta_j$ denote the angle for the circle $C_j$. Then, 
the sum of the angles for both circles is given by $\theta_j + \frac{\pi}{2}-\theta_{j+1}+ \frac{\pi}{2}=\pi$.
 Similarly, in the case $q_j<q_{j+1}$, the sum of the angles is $\pi$. This gives
 \[M_{\mathcal{\theta},I,\delta}(\mathcal{F})=\frac{1}{|I|}\sum_{j=1}^{N_I(Q)}(\theta_j+\theta_{j+1}) =\frac{2\pi}{|I|} N_I(Q)=\frac{12 }{\pi}Q^2 + \BigO{\log Q}.\]

\section{Dynamical Methods}\label{sec:dist} In this section, we prove Corollaries~\ref{thm:stat:ind} and~\ref{thm:stat:joint}. As we stated before, they are both consequences of the equidistribution of a certain family of measures on the region $$\Omega = \{(x,y) \in (0,1]^2: x+y > 1\}.$$
To see this, we write (with notation as in \S\ref{sec:dist:results})
$$ \frac{ \#\{1 \le j \le N_I(Q): F(x_j, y_j) \geq t\}}{N_I(Q)} = \rho_{Q, I} \left( F^{-1}\left(t, \infty\right) \right) $$

$$ \frac{ \#\{1 \le j \le N_I(Q): F_i(x_{j+n_i}, y_{j+n_i}) \geq t_i, 1 \le i \le k\}}{N_I(Q)} = \rho_{Q, I} \left( \left(\mathbf F \circ T\right)^{-\mathbf n} (t_1, \ldots, t_k) \right)$$

Now the results follow from applying Theorem~\ref{theorem:equi} to the above expressions.\qed\medskip
 
\begin{rem}[Shrinking Intervals] Versions of Corollaries~\ref{thm:stat:ind} and~\ref{thm:stat:joint} for \emph{shrinking intervals} $I_Q = [\alpha_Q, \beta_Q]$ where the difference $\beta_Q-\alpha_Q$ is permitted to tend to zero with $Q$ can be obtained by replacing Theorem~\ref{theorem:equi} with appropriate discretizations of results of Hejhal~\cite{Hejhal3:2000} and Str\"ombergsson~\cite{Strombergsson:2004}, see~\cite[Remark 1]{Athreya}.
\end{rem}

\section{Computations of the Tail Behavior of Geometric Statistics}\label{sec:geom} In this section, we compute the tails of our geometric statistics (\S\ref{sec:tails}).

\subsection{Euclidean distance}
For $ F(x,y)=\dfrac{1}{2x^2}+\dfrac{1}{2y^2}$, by Corollary \ref{thm:stat:ind},
\[
G_F(t)=m(F^{-1}(t,\infty))=m\left(\left\{(x,y)\in\Omega, \frac{1}{x^2}+\frac{1}{y^2}\ge 2t\right\}\right).\]
Define \[B_t:=\left\{(x,y)\in\Omega, \frac{1}{x^2}+\frac{1}{y^2}\ge 2t\right\}.\]
Therefore,
\[G_F(t)=\int_{B_t}~dm=2{\int\int}_{B_t}~dx~dy.\] 
Note that if $y<\frac{1}{\sqrt{2t}}$, the inequality 
$\dfrac{1}{x^2}+\dfrac{1}{y^2}\ge 2t$ automatically holds true. The region bounded by lines $0<y< 1/\sqrt{2t}$ and $1-y< x< 1$ is exactly an isosceles right triangle with area $\dfrac{1}{4t}$.
Let us assume in what follows that $y>\frac{1}{\sqrt{2t}}$. Then,
\begin{align*}
& x \leq \frac{1}{\sqrt{2t-\frac{1}{y^2}}}.
\end{align*}
This yields
\[1-y<x<\min \left\{1, \frac{1}{\sqrt{2t-\frac{1}{y^2}}}\right\}.\]
Therefore, $1-y<\frac{1}{\sqrt{2t-\frac{1}{y^2}}}$, which implies either
 $\alpha_1<y<\alpha_2$ or $\alpha_3<y<\alpha_4,$
where
\[ \alpha_1:=\frac{t-\sqrt{t^2+2t+2t(\sqrt{1+2t})}}{2t} \ ; \ \alpha_2:=\frac{t-\sqrt{t^2+2t-2t(\sqrt{1+2t})}}{2t},\]
and
\[\alpha_3:=\frac{t+\sqrt{t^2+2t-2t(\sqrt{1+2t})}}{2t}\ ; \ \alpha_4:=\frac{t+\sqrt{t^2+2t+2t(\sqrt{1+2t})}}{2t}. \]
 Also,
\[\alpha_1<0<\frac{1}{\sqrt{2t}}<\frac{1}{\sqrt{2t-1}}<\alpha_2<\alpha_3<1<\alpha_4.\numberthis \label{alphas}\]
Hence a point $(x,y)$ with $y>1/\sqrt{2t}$ belongs to $B_t$ if and only if 
\begin{align*}
& 1-y<x<\min \left\{1, \frac{1}{\sqrt{2t-\frac{1}{y^2}}}\right\} \text{ and }  y\in \left(\frac{1}{\sqrt{2t}},\alpha_2\right)\cup (\alpha_3,1). \numberthis \label{x-y}
\end{align*}
Now for the measure of the set $B_t$ minus the isosceles right triangle already discussed above, we consider the following two cases.\\ \\
\textit{Case I.} $\text{min}\left\{1, \frac{1}{\sqrt{2t-\frac{1}{y^2}}}\right\}=1$ i.e. $y\le\frac{1}{\sqrt{2t-1}}$. 
Employing \eqref{alphas} and \eqref{x-y}, we have
that $1-y<x<1$ and $\frac{1}{\sqrt{2t}}<y\le\frac{1}{\sqrt{2t-1}}$.\\
\textit{Case II.}
$\min \left\{1, \frac{1}{\sqrt{2t-\frac{1}{y^2}}}\right\}=\frac{1}{\sqrt{2t-\frac{1}{y^2}}}
$ i.e. $y\ge\frac{1}{\sqrt{2t-1}}$. 
Then, using \eqref{alphas} and \eqref{x-y}, we have that $1-y<x<\frac{1}{\sqrt{2t-\frac{1}{y^2}}}$ and either $\frac{1} {\sqrt{2t-1}}<y<\alpha_2$ or $\alpha_3<y<1.$

\noindent Combining the two cases above and adding the contribution from the isosceles right triangle, we obtain
\begin{align*}
G_F(t)&=\frac{1}{2t}+2\int_{\frac{1}{\sqrt{2t}}}^{\frac{1}{\sqrt{2t-1}}}\int_{1-y}^{1}~dx~dy+2\int_{\frac{1} {\sqrt{2t-1}}}^{\alpha_2}\int_{1-y}^{\frac{1}{\sqrt{2t-\frac{1}{y^2}}}}~dx~dy+ 2\int_{\alpha_3}^{1}\int_{1-y}^{\frac{1}{\sqrt{2t-\frac{1}{y^2}}}}~dx~dy\\&
=\frac{4t-1}{\sqrt{2 t-1}}-1-\frac{1}{t} \sqrt{\frac{1}{2 t-1}}+\frac{1}{t}\sqrt{t \left(t-2 \sqrt{2 t+1}+2\right)}\\&\quad +\frac{1}{t}\sqrt{t-\sqrt{2 t+1}-\sqrt{t \left(t-2 \sqrt{2 t+1}+2\right)}}\\&\quad-\frac{1}{t} \sqrt{t-\sqrt{2 t+1}+\sqrt{t \left(t-2 \sqrt{2 t+1}+2\right)}}.
\numberthis\label{euc2}
\end{align*}
This gives
\begin{align*}
G_F(t)\sim \frac{1}{t}.
\end{align*}

\subsection{Hyperbolic distance}
In this case $ F(x,y)=\dfrac{1}{2}\left(\dfrac{y}{x}+\dfrac{x}{y}\right)$ and
\[G_F(t)=m(C_t), \ \text{where} \ \ C_t:=\left\{(x,y)\in\Omega, \frac{1}{2}\left(\frac{y}{x}+\frac{x}{y}\right)\ge t\right\}.\]
\begin{align*}
\frac{1}{2}\left(\frac{y}{x}+\frac{x}{y}\right)\ge t \Rightarrow \ \text{either} \ x\ge y(t+\sqrt{t^2-1}) \ \text{or} \ x\le y(t-\sqrt{t^2-1})
\end{align*}
Since $x,y \in \Omega$, we note that either 
\[1-y<x<  \min\left\{1,y(t-\sqrt{t^2-1})\right\}\] or 
\[\max\left\{1-y,y(t+\sqrt{t^2-1})\right\}<x<1.\]
Depending on the bounds of $x$ and $y$, we have the following cases: \\
\textit{Case I.} $1-y<x<  \min\left\{1,y(t-\sqrt{t^2-1})\right\}. $\\
  As $y$ and $t-\sqrt{t^2-1}=\frac{1}{t+\sqrt{t^2+1}}<1$, we note that $y(t-\sqrt{t^2-1})<1$ and therefore \[\min\left\{1,y(t-\sqrt{t^2-1})\right\}=y(t-\sqrt{t^2-1}).\]This implies $y>\frac{1}{t+1-\sqrt{t^2-1}}$. Also, $y(t-\sqrt{t^2-1})<1$ implies $y<(t-\sqrt{t^2-1})^{-1}=t+\sqrt{t^2-1}$, which always holds since $y<1$ and $t \gg 1$. And so in this case the range for $x$ and $y$ is given by,
\begin{align*}
\frac{1}{t+1-\sqrt{t^2-1}}< y<1 \text{ and } \ 
1-y<x<  y(t-\sqrt{t^2-1}). \numberthis\label{euc5}
\end{align*}
\textit{Case II.} $\max\left\{1-y,y(t+\sqrt{t^2-1})\right\}<x<1$. 
Here, we consider two cases. \\
 \textit{Sub case I.} $\max\left\{1-y,y(t+\sqrt{t^2-1})\right\}=1-y.$ Then 
 \[y<\frac{1}{t+\sqrt{t^2-1}+1}\]
and the range for $x$ and $y$ is given by
\begin{align*}
0< y< \frac{1}{t+\sqrt{t^2-1}+1} \text{ and } 1-y< x< 1. \numberthis\label{dist1}
\end{align*}
\textit{Sub Case II.} When $\max\left\{1-y,y(t+\sqrt{t^2-1})\right\}=y(t+\sqrt{t^2-1})$. This implies
\[y>\frac{1}{t+\sqrt{t^2-1}+1} \text{ and } y(t+\sqrt{t^2-1})<1,\] and so we get,
\[\frac{1}{t+\sqrt{t^2-1}+1}<y<\frac{1}{(t+\sqrt{t^2-1})} \text{ and } y(t+\sqrt{t^2-1})< x< 1. \numberthis\label{dist2}\]
Combining all the above cases, from \eqref{euc5}, \eqref{dist1} and \eqref{dist2}, we have 
\begin{align*}
G_F(t)&=2\int_{\frac{1}{t-\sqrt{t^2-1}+1}}^1 \int_{1-y}^{y(t-\sqrt{t^2-1})}~dx~dy+2\int_0^{\frac{1}{t+\sqrt{t^2-1}+1}}\int_{1-y}^{1}~dx~dy\\&\quad+
2\int_{\frac{1}{t+\sqrt{t^2-1}+1}}^{\frac{1}{t+\sqrt{t^2-1}}}\int_{y(t+\sqrt{t^2-1})}^{1}~dx~dy\\
&=\frac{1}{(t+\sqrt{t^2-1})(t+1+\sqrt{t^2-1})}+\frac{1}{(t+\sqrt{t^2-1}+1)^2}\\&\quad +\frac{1}{(t+\sqrt{t^2-1})(t+\sqrt{t^2-1}+1)^2}\\&=\frac{2}{(t+\sqrt{t^2-1})(t+\sqrt{t^2-1}+1)}.
\end{align*}
Thus, \[G_F(t)\sim\frac{1}{2t^2}.\]

\begin{figure}[h!]
\begin{subfigure}[t]{.45\textwidth}
\includegraphics[scale=0.5]{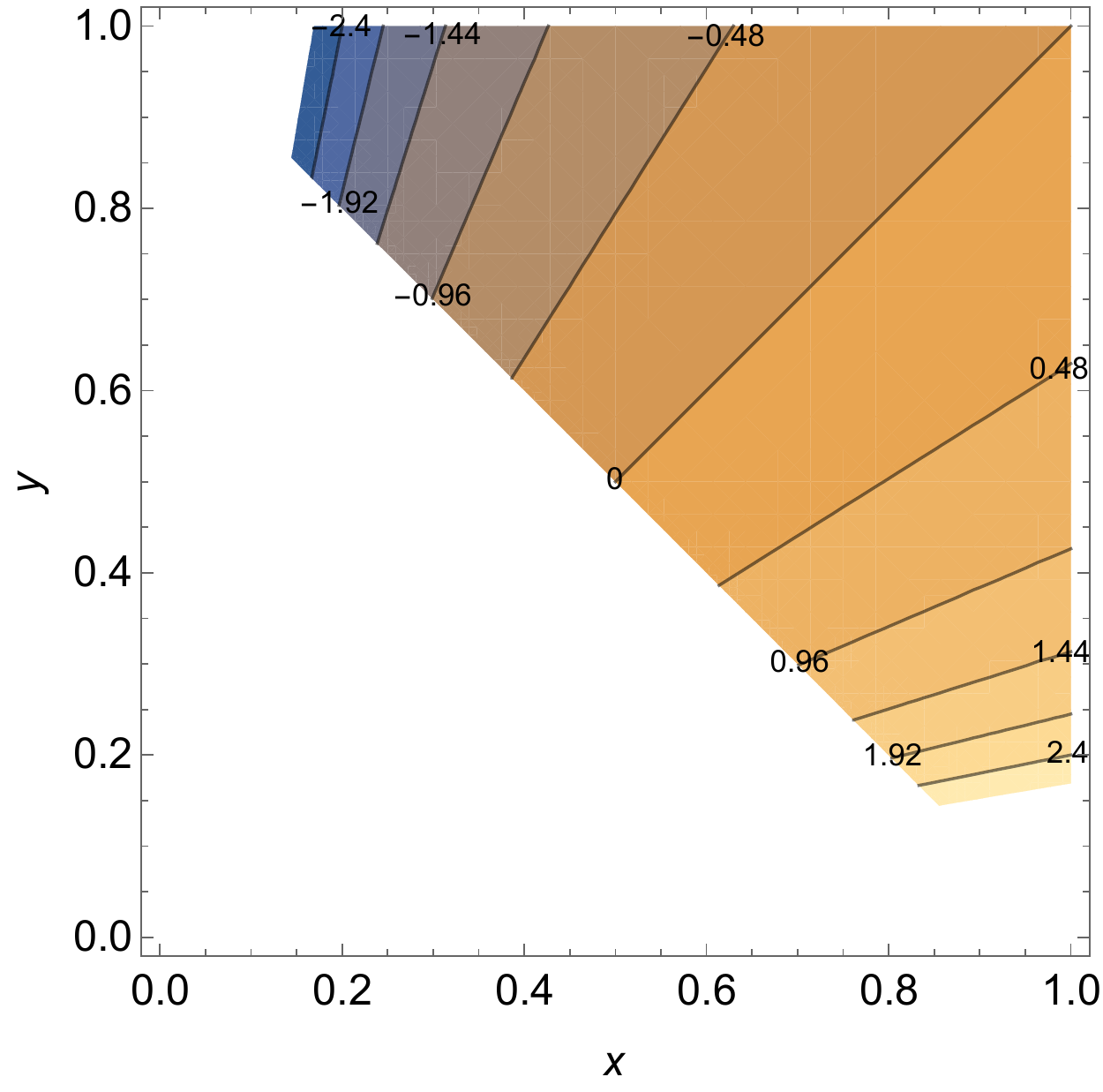}
\caption{Level curves for geometric statistic for Euclidean slope in the region $\Omega$.}
\label{fig:B}
\end{subfigure}
\begin{subfigure}[t]{.45\textwidth}
\includegraphics[scale=0.5]{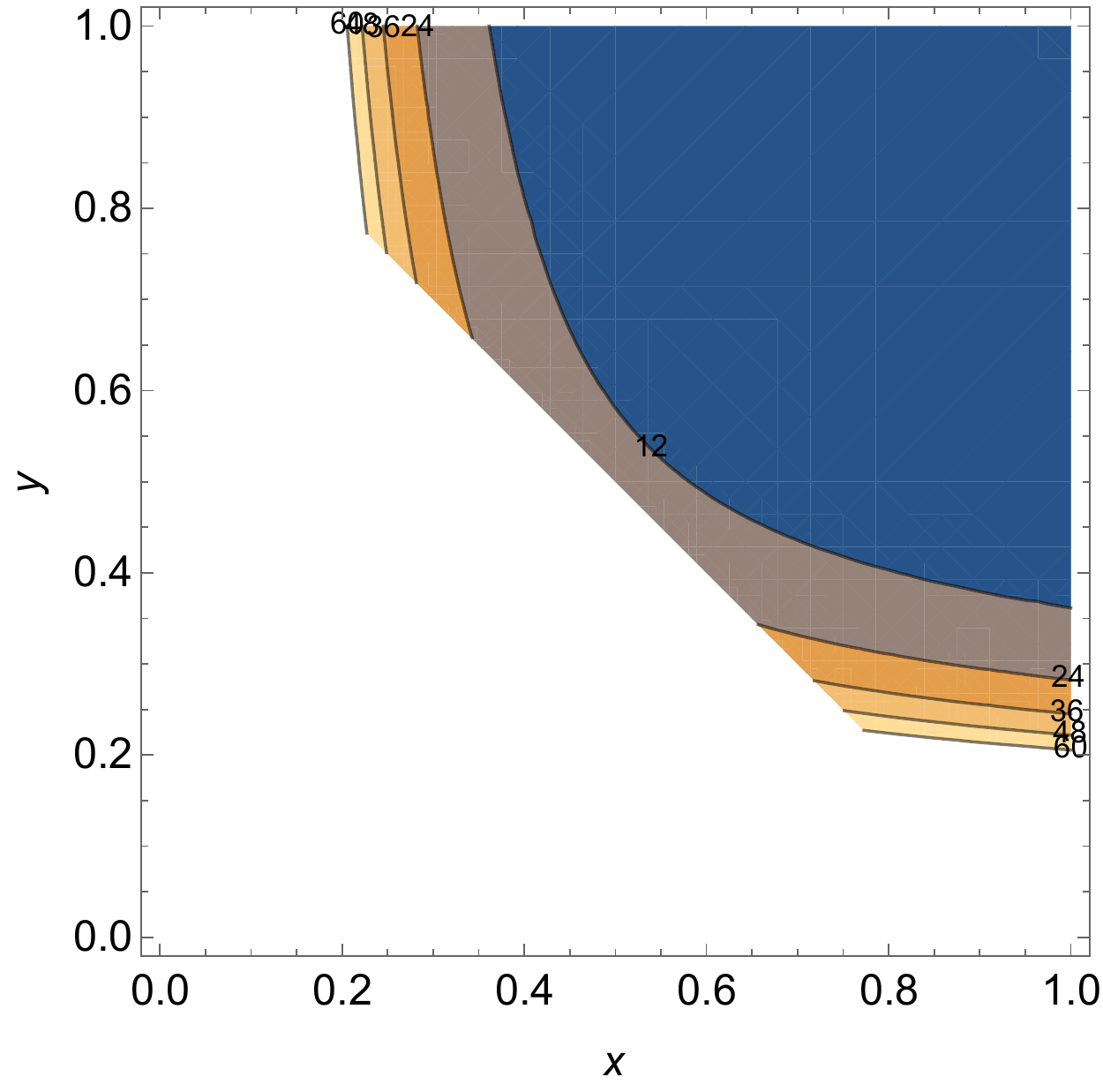}
\caption{Level curves for geometric statistic for Euclidean area in the region $\Omega$.}
\label{fig:A}
\end{subfigure}
\end{figure}

The tails for the geometric statistics for Euclidean slope, Euclidean area and hyperbolic angle can be computed in a similar manner. 

\begin{figure}[h!]
\begin{subfigure}[t]{.43\textwidth}
\includegraphics[scale=0.5]{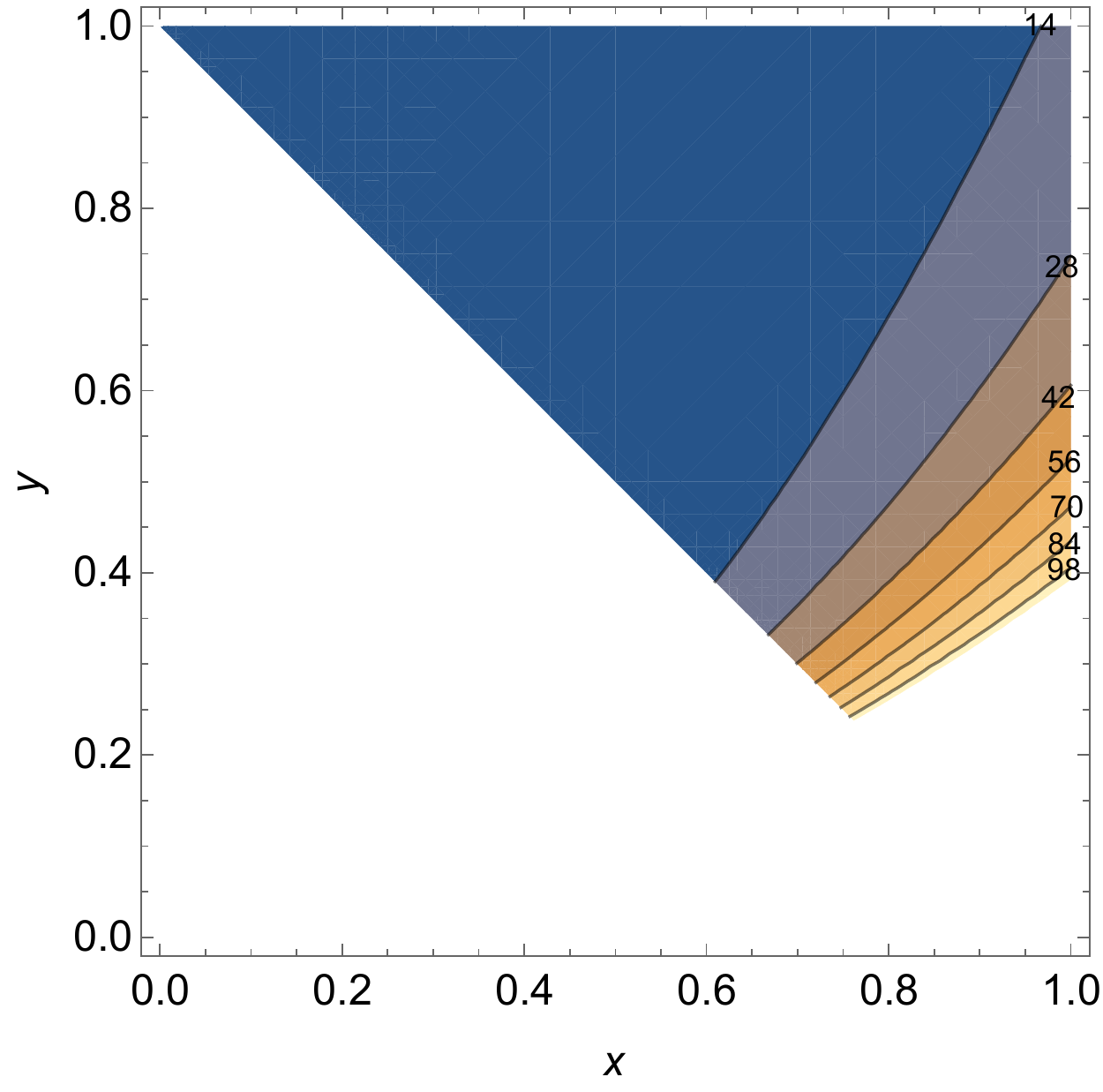}
\caption{Level curves for geometric statistic for hyperbolic angle.}
\label{fig:C}
\end{subfigure}
\end{figure}

\begin{ack}
The authors are grateful to the referee for many useful suggestions and comments.
\end{ack}

\end{document}